\documentclass[reqno,11pt,centertags,draft]{amsart}

\usepackage {amsfonts}
\usepackage[english]{babel}
\def\ca{{\mathcal{A}}}
\def\css{{\mathcal{S}}}

\def\a{\alpha}
\def\b{\beta}
\def\l{\lambda}

\def\o{\omega}

\def\R{{\mathbb R}}

\def\D{{\mathbb D}}
\def\T{{\mathbb T}}

\def\g{\gamma}

\def\ep{\varepsilon}

\def\ovl{\overline}

\def\dist{{\rm dist}}

\newtheorem{theorem}{Theorem}

\newtheorem{lemma}[theorem]{Lemma}
\newtheorem{corollary}[theorem]{Corollary}

\begin{document}

\title[Critical points]
{On critical points of Blaschke products}
\author[S. Favorov and L. Golinskii]{S. Favorov and L. Golinskii}

\address{Mathematical School, Kharkov National University, 4 Svobody sq.,
Kharkov 61077, Ukraine} \email{Sergey.Ju.Favorov@univer.kharkov.ua}

\address{Mathematics Division, Institute for Low Temperature Physics and
Engineering, 47 Lenin ave., Kharkov 61103, Ukraine}
\email{leonid.golinskii@gmail.com}

\date{\today}


\begin{abstract}
We obtain an upper bound for the derivative of a Blaschke product,
whose zeros lie in a certain Stolz-type region. We show that
the derivative belongs to the space of analytic functions in the unit disk, introduced recently in \cite{FG}. As an outcome, we obtain a Blaschke-type condition for critical points of such Blaschke products.
\end{abstract}

\maketitle

\section{Introduction}

Given a sequence $\{z_n\}\subset\D$ subject to the Blaschke
condition
\begin{equation}\label{blacon}
\a:=\sum_{n=1}^\infty (1-|z_n|)<\infty,
\end{equation}
let
\begin{equation*}
B(z)=\prod_{n=1}^\infty b_n(z), \qquad b_n(z)=\frac{\ovl
z_n}{|z_n|}\frac{z_n-z}{1-\ovl z_n z}
\end{equation*}
be a Blaschke product with the zero set $Z(B)=\{z_n\}$. With no loss
of generality we will assume that $B(0)\not=0$.

One of the central problems with Blaschke products is that of the
membership of their derivatives in classical function spaces, $B'\in
X$. There is a vast literature on the problem, starting from
investigations of P. Ahern and his collaborators \cite{A79, A83,
AC74, AC76} and D. Protas \cite{P73} in 1970s, up to quite recent
results of the Spanish school \cite{GP06, GPV07, GPV08}, see also
\cite{MS09, P10}. The above mentioned spaces $X$ are primarily the
Hardy spaces $H^p$, the Bergman spaces $A^p$, the Banach envelopes
of the Hardy spaces $B^p$ etc. Recall the definition of the Bergman
spaces $A^p$, $p>0$:
$$ A^p=\{f\in\ca(\D): \ \ \int\limits_\D |f(z)|^p\,dxdy<\infty\},
\quad z=x+iy.
$$

In this paper we will add to the list some new spaces
$X=\ca(E,\rho)$ of analytic functions in the unit disk, introduced
recently in \cite{FG}. Given a closed set $E=\ovl E\subset\T$ and
$\rho>0$, we say that an analytic function $f$ belongs to
$\ca(E,\rho)$ if
\begin{equation*}
|f(z)|\le C_1\exp\left(\frac{C_2}{d^\rho(z,E)}\right)\,, \quad
d(z,E)=\dist(z,E)
\end{equation*}
is the distance from $z\in\D$ to $E$, $C_{1,2}$ are positive
constants.

The simplest and most general result drops out immediately from the
Schwarz--Pick lemma for functions $g$ from the unit ball of
$H^\infty$:
$$ |g'(z)|\le\frac{1-|g(z)|^2}{1-|z|^2}\,, $$
and states, that $g'\in A^p$ for all $0<p<1$. The result is sharp:
there exists a Blaschke product $B$ such that $B'\notin A^1$ (W.
Rudin).

To proceed further, one should impose some additional restrictions
either on absolute values $|z_n|$, stronger than \eqref{blacon}, or
on location, distribution of arguments of zeros $z_n$ etc. For
instance, a typical result of the first type is due to Protas
\cite{P73}:
$$ \sum_{n=1}^\infty (1-|z_n|)^r<\infty, \quad 0<r<\frac12 \ \
\Rightarrow \ \ B'\in H^{1-r}, $$ and $\frac12$ is sharp.

We are more interested in the second direction, related to the
location of zeros. A typical assumption here is that $Z(B)$ belongs
to certain regions inside the unit disk.

Let $t\in\T$, $\g\ge1$. Following \cite{C62, AC74, GPV08} we
introduce regions
\begin{equation}\label{genstol}
R(t,\g,K):=\{\l\in\D: |t-\l|^{\g}\le K(1-|\l|)\}, \quad K\ge1.
\end{equation}
For $\g=1$, $K>1$ this is the standard Stolz angle. When $\g>1$ the
region touches the circle $\T$ at the vertex $t$ with the power
degree of tangency. The following result claims that $B'$ belongs to
$H^p$ or $A^p$ as soon as $Z(B)\subset R(t,1,K)$.

{\bf Theorem A}. Let $Z(B)\subset R(t,1,K)$. Then
\begin{enumerate}
\item $B'\in H^p$, $p<\frac12$, and $\frac12$ is sharp;
\smallskip
\item $B'\in A^p$, $p<\frac32$, and $\frac32$ is sharp.
\end{enumerate}

\noindent The first statement is proved in \cite[Theorem
2.3]{GPV07}, for the second one see \cite{GPV07, GP06}. For related
results in the case $Z(B)\subset R(t,\g,K)$ with $\g>1$, see
\cite[Section 3]{GPV08}.

We study the same problem for more general Stolz--type regions.

A function $\phi$ on the right half-line will be called a {\it model
function}, if it is nonnegative, continuous and increasing, and
\begin{equation}\label{model}
\phi(x)\le Cx, \qquad x\ge0, \quad C=C(\phi)>0.
\end{equation}
We define a {\it Stolz angle associated with a model function
$\phi$} with the vertex at $t\in\T$ as
\begin{equation}\label{expstol}
\css_\phi(t,K)=\css(t,K):=\{\l\in\D: \phi(|t-\l|)\le K(1-|\l|)\},
\quad K>0.
\end{equation}
Since $|t-\l|\le2$ for $t,\l\in\ovl\D$, it is clear that regions
\eqref{genstol} are of the form \eqref{expstol} for an appropriate
$\phi$. Precisely, one can put $\phi(x)=x^\g$ for $0\le x\le 2$, 
and $\phi(x)=2^{\g-1}x$ for $x\ge2$.
Next, given a closed set $E=\ovl E\subset\T$ we define a
{\it Stolz region, associated with a model function $\phi$ and the set
$E$}, as
\begin{equation*}
\css(E,K):=\{\l\in\D: \phi(d(\l,E))\le
K(1-|\l|)\}=\bigcup\limits_{t\in E} \css(t,K).
\end{equation*}

Here is our main result.
\begin{theorem}\label{mainres}
Let $B$ be a Blaschke product such that $Z(B)\subset\css(E,K)$. Then
\begin{equation}\label{bla4}
|B'(z)|\le
2(2C+K)^2
\sum_{n=1}^\infty(1-|z_n|)\,\phi^{-2}\left(\frac{d(z,E)}6\right)\,.
\end{equation}
\end{theorem}

For the standard Stolz angle and $E=\{t\}$ we take $\phi(x)=x$, so
$$ |B'(z)|\le \frac{C_3}{|t-z|^2}, $$
and part (1) in Theorem A follows. Similarly, for the region $R(t,\g,K)$, $\g>1$, \eqref{bla4} implies
$$ |B'(z)|\le \frac{C_4}{|t-z|^{2\g}}, $$
and we come to the following result (cf. \cite[Remark 1]{GPV08}).
\begin{corollary}
If $Z(B)\subset R(t,\g,K)$, $\g>1$, then $B'\in H^p$ for all $p<1/2\g$.
\end{corollary}

\medskip

We are particularly interested in the model function
$\phi(x)=\exp\{-x^{-\rho}\}$, $\rho>0$. In this case \eqref{bla4}
says that $B'\in\ca(E,\rho)$.

Denote $Z(B')=\{z_n'\}$ the zero set of $B'$. Each result of the
form $B'\in X$ provides some information about the critical points
of $B$ (zeros of $B'$), as long as the information about zero sets
of functions from $X$ is available. The most general condition
applied to an arbitrary Blaschke product arises from the fact that
$B'\in A^p$, $p<1$, so (cf. \cite[Theorem 4.7]{HKZ})
$$ \sum_{n=1}^\infty
\frac{1-|z_n'|}{\left(\log\frac1{1-|z_n'|}\right)^{1+\ep}}<\infty,
\qquad \forall\ep>0. $$
 On the other hand there are Blaschke products $B$ such that $\sum
1-|z_n'|=\infty$ (see, e.g., \cite{Pom}).

A Blaschke--type condition for zeros of functions from $\ca(E,\rho)$
is given in a recent paper \cite{FG}. To present its main result we
define, following P. Ahern and D. Clark \cite[p.113]{AC76}, the type
$\b(E)$ of a closed subset $E$ of the unit circle as
$$ \b(E):=\sup\{\b\in\R:\ |E_x|=O(x^\b), \ x\to 0\}, $$
where
$$ E_x:=\{t\in\T:\ d(t,E)<x\}, \quad x>0, $$
is an $x$-neighborhood of $E$, $|E_x|$ its normalized Lebesgue
measure. For the equivalent definition and properties of the type
see also \cite{FG, FG10}.

\begin{theorem}\label{mainres1}
Given a closed set $E\subset\T$ and a Blaschke product $B$, assume
that $Z(B)\subset\css_\phi(E,K)$, $\phi(x)=\exp\{-x^{-\rho}\}$,
$\rho>0$. Then
\begin{equation*}
\sum_{n=1}^\infty (1-|z_n'|)d^{(\rho-\b(E)+\ep)_+}(z_n',E)<\infty,
\qquad \forall\ep>0,
\end{equation*}
$\b(E)$ is the type of $E$, $(a)_+=\max(a,0)$.
\end{theorem}
It is clear that $\b(E)=1$ for each finite set $E$.
\begin{corollary}
Let $Z(B)\subset\css_\phi(t,K)$ with the same $\phi$. Then
\begin{equation*}
\sum_{n=1}^\infty (1-|z_n'|)|t-z_n'|^{(\rho-1+\ep)_+}<\infty, \qquad
\forall\ep>0.
\end{equation*}
\end{corollary}

The authors thank the referee for a number of comments that improved the paper write-up.

\section{Main results}

A model function $\phi$ is nonnegative
and increasing, so for all $x,y,u\ge 0$
\begin{equation}\label{phi}
\phi\left(\frac{x+y+u}3\right)\le \phi(x)+\phi(y)+\phi(u).
\end{equation}

We begin with the following result, which is similar to Vinogradov's
lemma from \cite{V94}.
\begin{lemma}\label{mainlem}
Let $z\in\ovl\D$, $t\in\T$ and $\l\in \css(t,K)$. Then
\begin{equation}\label{mainbound}
\frac1{|1-\ovl\l z|}\, \phi\left(\frac{\left
|t-z|\l|\right|}3\right)\le 2C+K.
\end{equation}\end{lemma}
\begin{proof}
With no loss of generality we assume that $t=1$. Since
\begin{equation*}
\begin{split}
\left|1-z|\l|\right|&=\left|1-\ovl\l z+z(\ovl\l-|\l|)\right|\le
|1-\ovl\l z|+|\ovl\l-|\l|| \\ &\le |1-\ovl\l z|+(1-|\l|)+|1-\l|,
\end{split}
\end{equation*}
then by \eqref{phi}
\begin{equation*}
\phi\left(\frac{\left |1-z|\l|\right|}3\right)\le
\phi(\left|1-\ovl\l z \right|)+\phi(1-|\l|)+\phi(|1-\l|),
\end{equation*} so
$$ \frac1{|1-\ovl\l z|}\,
\phi\left(\frac{\left |1-z|\l|\right|}3\right)\le A_1+A_2+A_3. $$
By \eqref{model}, for the first two terms we have
$$ A_1=\frac{\phi(\left
|1-\ovl\l z\right|)}{|1-\ovl\l z|}\le C, \quad
A_2\le \frac{\phi(1-|\l|)}{1-|\l|}\le C. $$
As for the third one,
\begin{equation*}
A_3
\le\frac{\phi(|1-\l|)}{1-|\l|}\le K
\end{equation*}
according to the assumption $\l\in \css(1,K)$. The proof is complete.
\end{proof}

Our main result gives a bound for the derivative $B'$ in the case
when the zero set $Z=Z(B)\subset\css(E,K)$.

\noindent {\it Proof of Theorem \ref{mainres}}.
Denote $Z(t):=Z\bigcap \css(t,K)$, $t\in E$. Then there is an
at most countable set $\{t_k\}_{k=1}^\o$, $\o\le\infty$, $t_k\in E$,
so that $z_k\in Z(t_k)$, and $Z=\bigcup_k Z(t_k)$. It is clear that
there is a disjoint decomposition
$$ Z=\bigcup_k Z_k, \quad Z_k\not=\emptyset, \qquad Z_k\subset Z(t_k),
\qquad Z_j\bigcap Z_k=\emptyset, \quad j\not=k. $$ Let us label the
set $Z$ in such a way that
$$ Z=\{z_{kj}\}, \quad k=1,2,\ldots,\o, \quad j=1,2,\ldots\o_k, \quad
\{z_{kj}\}_{j=1}^{\o_k}\subset Z_k. $$

We proceed with the expression
$$ B'(z)=\sum_{j=1}^\infty b_n'(z)B_n(z), \qquad
B_n(z)=\frac{B(z)}{b_n(z)}\,, $$ so
\begin{equation*}
|B'(z)|\le \sum_{n=1}^\infty \frac{1-|z_n|^2}{|1-\ovl z_n
z|^2}\,|B_n(z)|\le \sum_{n=1}^\infty \frac{1-|z_n|^2}{|1-\ovl z_n
z|^2}=\sum_{k=1}^\o\,\sum_{j=1}^{\o_k}\frac{1-|z_{kj}|^2}{|1-\ovl z_{kj}
z|^2}\,.
\end{equation*}
Note that $|t-z|\le 2\left|t-z|\l|\right|$ for all $z,\l\in\D$ and
$t\in\T$. Indeed,
$$ \left|t-z|\l|\right|\ge 1-|z\l|\ge 1-|\l| $$
and
$$ |t-z|\le \left|t-z|\l|\right|+|z|(1-|\l|)\le 2\left|t-z|\l|\right|, $$
as claimed. Hence
$$
\phi\left(\frac{|t-z|}6\right)\le
\phi\left(\frac{|t-z|\l||}3\right)\,,
$$ and by \eqref{mainbound}
\begin{equation*}
\begin{split}
\phi^2\left(\frac{d(z,E)}6\right)\,|B'(z)| &\le
\sum_{k=1}^\o\,\sum_{j=1}^{\o_k}\frac{1-|z_{kj}|^2}{|1-\ovl z_{kj}
z|^2}\,\phi^2\left(\frac{|t_k-z|}6\right)
\\
&\le \sum_{k=1}^\o\,\sum_{j=1}^{\o_k}\frac{1-|z_{kj}|^2}{|1-\ovl
z_{kj} z|^2}\,\phi^2\left(\frac{|t_k-z|z_{kj}||}3\right)
\\ &\le 2(2C+K)^2\sum_{n=1}^\infty (1-|z_n|),
\end{split}\end{equation*}
which is \eqref{bla4}. The proof is complete.

Theorem \ref{mainres1} is a direct consequence of Theorem
\ref{mainres} and \cite[Theorem 3]{FG}.


\begin{thebibliography}{9999}

\bibitem{A79}
P. Ahern, {\it The mean modulus of the derivateve of an inner
function}, Indiana Univ. Math. J., {\bf 28} (1979), 311--347.

\bibitem{A83}
P. Ahern, {\it The Poisson integral of a singular measure}, Can. J.
Math. {\bf 35} (1983), 735--749.

\bibitem{AC74}
P. Ahern and D. Clark, {\it On inner functions with $H^p$
derivatives}, Michigan Math. J., {\bf 21} (1974), 115--127.

\bibitem{AC76}
P. Ahern and D. Clark, {\it On inner functions with $B^p$
derivatives}, Michigan Math. J., {\bf 23} (1976), 107--118.


\bibitem{C62}
G. Cargo, {\it Angular and tangential limits of Blaschke products
and their successive derivatives}, Canad. J. Math., {\bf 14} (1962),
334--348.

\bibitem{FG}
S.~Favorov and L.~Golinskii, {\it A Blaschke-type condition for
analytic and subharmonic functions and application to contraction
operators}, Amer. Math. Soc. Transl., {\bf 226} (2) (2009), 37--47.

\bibitem{FG10}
S.~Favorov and L.~Golinskii, {\it Blaschke-type conditions for
analytic functions in the unit disk: inverse problems and local
analogs}, preprint arXive:1007.3020v1 [math CV] 18 Jul 2010.

\bibitem{GP06}
D. Girela and J. Pel\'aes, {\it On the membership in Bergman spaces
of the derivative of A Blaschke product with zeros in a Stolz
domain}, Canad. Math. Bull., {\bf 49} (3) (2006), 381--388.

\bibitem{GPV07}
D. Girela, J. Pel\'aes, and D. Vucoti\'c, {\it Integrability of the
derivative of a Blaschke product}, Proc. Edinburgh Math. Soc., {\bf
50} (2007), 673--687.

\bibitem{GPV08}
D. Girela, J. Pel\'aes, and D. Vucoti\'c, {\it Interpolating
Blaschke products: Stolz and tangential approach regions}, Constr.
Approx., {\bf 27} (2008), 203--216.

\bibitem{HKZ}
H.  Hedenmalm, B. Korenblum, K. Zhu, {\it Theory of Bergman spaces.}
Graduate Texts in Mathematics, vol. 199. Springer-Verlag, New York,
2000.


\bibitem{MS09}
J. Mashreghi and M. Shabankhah, {\it Integral means of the
logarithmic derivative of Blaschke products}, Comp. Meth. Func.
Theory, {\bf 9} (2) (2009), 421--433.

\bibitem{Pom}
Ch. Pommerenke, {\it On the Green's function of Fuchsian groups},
Ann. Acad. Sci. Fenn., {\bf 2} (1976), 409--427.

\bibitem{P73}
D. Protas, {\it Blaschke products with derivative in $H^p$ and
$B^p$}, Michigan Math. J., {\bf 20} (1973), 393--396.

\bibitem{P10}
D. Protas, {\it Blaschke products with derivatives in function
spaces}, preprint arXive:1001.5098v2 [math.CV] 10 Jun 2010.

\bibitem{V94}
S. A. Vinogradov, {\it Multiplication and division in the space of
analytic functions with area integrable derivative, and some related
spaces}, Zap. Nauchn. Semin. POMI, {\bf 222} (1994), 45--77.
(Russian)



\end{thebibliography}
\end{document}